\documentclass[a4paper]{amsart}

\usepackage{amsmath, amsfonts, amssymb, amsthm, amscd}
\usepackage{graphicx, color}
\usepackage{pinlabel}

\theoremstyle{plain}
\newtheorem{theorem}{Theorem}[section]
\newtheorem{conjecture}[theorem]{Conjecture}
\newtheorem{proposition}[theorem]{Proposition}

\newtheorem{lemma}[theorem]{Lemma}

\newtheorem*{mob_theorem}{M\"obius-like Theorem}
\newtheorem*{compactification_theorem}{Compactification Theorem}
\newtheorem*{conj_theorem}{Conjugacy Theorem}

\theoremstyle{definition}
\newtheorem{definition}[theorem]{Definition}
\newtheorem{construction}[theorem]{Construction}

\theoremstyle{remark}

\newtheorem{example}{Example}

\newcommand{\bR}{\mathbb{R}}

\newcommand{\bZ}{\mathbb{Z}}
\newcommand{\bH}{\mathbb{H}}

\newcommand{\cD}{\mathcal{D}}
\newcommand{\cE}{\mathcal{E}}

\newcommand{\cU}{\mathcal{U}}

\newcommand{\orb}{P}

\newcommand{\rF}{\mathfrak{F}}

\title{Quasigeodesic flows and M\"obius-like groups}
\author{Steven Frankel}
\address{DPMMS \\ University of Cambridge \\ Cambridge CB3 0WA \\ United Kingdom}
\email{sf451@cam.ac.uk}
\date{\today}

\begin{document}

\begin{abstract}
If $M$ is a hyperbolic 3-manifold with a quasigeodesic flow then we show that $\pi_1(M)$ acts in a natural way on a closed disc by homeomorphisms. Consequently, such a flow either has a closed orbit or the action on the boundary circle is M\"obius-like but not conjugate into $PSL(2, \bR)$. We conjecture that the latter possibility cannot occur.
\end{abstract}

\maketitle
\tableofcontents

\section{Introduction}
\subsection{Background and motivation}
In 1950, Seifert asked whether every nonsingular flow on the 3-sphere has a closed orbit \cite{Seifert}. Schweitzer gave a counterexample in 1974 and showed more generally that every homotopy class of nonsingular flows on a 3-manifold contains a $C^1$ representative with no closed orbits \cite{Schweizer}. Schweitzer's examples were generalized considerably and it is known that the flows can be taken to be smooth \cite{KKuperberg} or volume-preserving \cite{GKuperberg}.

On the other hand, Taubes' 2007 proof of the 3-dimensional Weinstein conjecture shows that flows satisfying certain geometric constraints must have closed orbits \cite{Taubes}. Explicitly, Taubes showed that every Reeb vector field on a closed 3-manifold has a closed orbit. Reeb flows are \emph{geodesible}, i.e. there is a Riemannian metric in which the flowlines are geodesics. Complementary to this result, though by different methods, Rechtman showed in 2010 that the only geodesible real analytic flows on closed 3-manifolds that contain no closed orbits are on torus bundles over the circle with reducible monodromy \cite{Rechtman}.

Geodesibility is a local condition, and furthermore one that is not stable under perturbations. By contrast, a nonsingular flow is said to be \emph{quasigeodesic} if the flowlines of the flow pulled back to the universal cover are quasigeodesics. This is a macroscopic condition, and when the ambient 3-manifold is hyperbolic it is a stable condition under $C^0$ perturbations; this stability is for global topological reasons and not because the flow itself is structurally stable (which it will not typically be).

Calegari conjectured in 2006 that quasigeodesic flows on closed hyperbolic 3-manifolds should all have closed orbits, and moreover that every homotopy class of quasigeodesic flow should contain a pseudo-Anosov representative that is unique up to isotopy. Pseudo-Anosov flows are hyperbolic and therefore structurally stable, so this conjecture implies that one should be able to deduce the existence of closed orbits from the dynamics of the fundamental group on the orbit space in the universal cover.

Our paper is devoted to fleshing out some aspects of Calegari's conjectural program. We are able to find conditions that guarantee the existence of a closed orbit for a quasigeodesic flow on a closed hyperbolic 3-manifold expressed in terms of the action of the fundamental group on an associated ``universal circle''.

\subsection{Statement of results}
A quasigeodesic is a map from $\bR$ to a metric space $X$ with bounded distortion. That is, distances as measured in $X$ and in $\bR$ are comparable on the large scale up to a multiplicative constant (see Section~\ref{sec:Quasigeodesics} for a precise definition).

A nonsingular flow $\rF$ on a 3-manifold $M$ is said to be \emph{quasigeodesic} if, after pulling back to a flow $\widetilde{\rF}$ on the universal cover $\widetilde{M}$, the flowlines of $\widetilde{\rF}$ are quasigeodesics in $\widetilde{M}$. If $M$ is compact, this depends not on a choice of metric on $M$ but only on the isotopy class of $\rF$.

We restrict attention in the sequel to the generic situation that $M$ is a closed hyperbolic 3-manifold; equivalently that $M$ is irreducible and the fundamental group of $M$ is infinite and does not contain $\bZ \oplus \bZ$.

If $\rF$ is a quasigeodesic flow on such a 3-manifold then the orbit space $\orb$ is homeomorphic to a plane and the fundamental group $\pi_1(M)$ acts on $\orb$ as the quotient of the deck group action on $\widetilde{M}$. The existence of a closed orbit is equivalent to the existence of a fixed point in $\orb$ for some nontrivial element of $\pi_1(M)$.

Calegari showed that the action of $\pi_1(M)$ on $\orb$ induces an action on a ``universal circle'' $S_u$ that is homeomorphic to $S^1$. Our first main result is a natural compactification of $\orb$ as a closed disc $\overline{\orb}$ in such a way that the boundary of $\overline{\orb}$ is the universal circle. This answers a question of Calegari in \cite{Calegari}.

\begin{compactification_theorem}
There is a natural compactification $\overline{\orb}$ of $\orb$ homeomorphic to the closed disc so that $\partial \overline{\orb} = S_u$. The action of $\pi_1(M)$ on $\orb$ extends to $\overline{\orb}$ and restricts to the universal circle action on $\partial \overline{\orb}$.
\end{compactification_theorem}

This will follow from Section~\ref{sec:EndCompactification} where we prove a more general result: If $\cD$ is a decomposition of the plane $P$ into closed, unbounded sets then there is a natural compactification of $P$ as a closed disc $\overline{P}$ such that the ends of each decomposition element appear as points in $\partial \overline{P}$ and distinct ends remain as separated as possible.

A group of homeomorphisms of the circle is said to be \emph{M\"obius-like} if each element is topologically conjugate to a M\"obius trasformation, i.e. conjugate to an element of $PSL(2, \bR)$ acting in the standard way on $\bR P^1$. A M\"obius-like group is \emph{rotationless} if each element has a fixed point. Examples of M\"obius-like actions are those of \emph{convergence groups} but it is known by the work of Casson-Jungreis, Gabai, Mess, Tukia, et. al. that convergence groups are \emph{globally} conjugate to subgroups of $PSL(2, \bR)$.  Our second main result concerns the action of $\pi_1(M)$ on the circle $\partial \overline{\orb}$.

\begin{mob_theorem}
Let $\rF$ be a quasigeodesic flow on a closed hyperbolic 3-manifold $M$. Suppose that $\pi_1(M)$ acting on the universal circle $S_u$ is not a rotationless M\"obius-like group. Then $\rF$ has a closed orbit.
\end{mob_theorem}

As a counterpoint to this theorem we have the following.

\begin{conj_theorem}
Let $\rF$ be a quasigeodesic flow on a closed hyperbolic 3-manifold $M$. Then the action of $\pi_1(M)$ on $\partial \overline{\orb}$ is not conjugate into $PSL(2, \bR)$.
\end{conj_theorem}

\subsection{Future directions}
The only known examples of M\"obius-like groups which are not conjugate into $PSL(2, \bR)$ were constructed by Kovacevic \cite{Kovacevic}. However, it seems unlikely that the actions arising from quasigeodesic flows could be of this type. Therefore we conjecture:

\begin{conjecture}
The action of $\pi_1(M)$ on the universal circle $\partial \overline{\orb}$ is not M\"obius-like.
\end{conjecture}

A positive answer to this conjecture implies the aforementioned conjecture that all quasigeodesic flows on closed hyperbolic 3-manifolds have closed orbits.

One possible route to this conjecture is as follows. The quasigeodesic property is stable under perturbations, so $\rF$ can be approximated by another quasigeodesic flow $\rF'$ with closed orbits. The corresponding fixed points in $\orb'$ imply that the action on $\partial \overline{\orb'}$ is not M\"obius-like. If the action on $\partial \overline{\orb'}$ could be shown to be structurally stable then we would deduce that $\rF$ had a closed orbit as well. We plan to pursue this idea in a future paper.

\subsection{Acknowledgements}
I would like to thank Danny Calegari for introducing me to this project, for countless helpful conversations, and for his extensive input on the exposition and organization of this paper. Thanks to Matt Day, Sergio Fenley, Nikolai Makarov, Curt McMullen, and Henry Wilton for helpful conversations and correspondence.

\section{Quasigeodesics and flows}
\subsection{Quasigeodesics} \label{sec:Quasigeodesics}
\begin{definition}
Let $k, \epsilon$ be non-negative constants. A curve $\gamma: \bR \to X$ in a metric space $(X, d)$ is a \emph{$(k, \epsilon)$-quasigeodesic} if we have
\[ 1/k \cdot d(\gamma(x), \gamma(y)) - \epsilon \leq |x - y| \leq k \cdot d(\gamma(x), \gamma(y)) + \epsilon \]
for all $x, y \in \bR$. A curve is called a \emph{quasigeodesic} if it is a $(k, \epsilon)$-quasigeodesic for some constants $k, \epsilon$.
\end{definition}

If $\gamma$ is parametrized by arc length and $X$ is a geodesic space then the left inequality is always satisfied.

The property of being a quasigeodesic is invariant under bilipschitz reparametrization, though the constants may change. In hyperbolic space, quasigeodesity can also be reformulated as a local condition.

\begin{lemma}[Gromov, see \cite{Calegari}, Lemma 3.9 and \cite{Gromov}]
A curve $\gamma: \bR \to \bH^3$ is called a \emph{$c$-local $k$-quasigeodesic} if $d(\gamma(x), \gamma(y)) \geq \frac{|x - y|}{k} - k$ for all $x, y \in \bR$ with $|x - y| < c$.

For every $k \geq 1$ there is a universal constant $c(k)$ such that every $c(k)$-local $k$-quasigeodesic is a $(2k, 2k)$-quasigeodesic.
\end{lemma}

Quasigeodesics in hyperbolic space  are qualitatively similar to geodesics:
\begin{proposition}[see \cite{Kapovich} or \cite{Bridson_Haefliger}, pp. 399-404] \label{prop:QuasisInH}
Let $\gamma$ be a quasigeodesic in $\bH^3$. Then $\gamma$ has distinct, well-defined endpoints in the sphere at infinity, i.e. there are distinct points $p, q \in \partial \bH^3$ such that 
				\[ \lim_{t \to \infty} \gamma(t) = p, \text{ and} \]
				\[ \lim_{t \to -\infty} \gamma(t) = q. \]
In addition, there are universal constants $C_{k, \epsilon}$ depending only on $k$ and $\epsilon$ such that each $(k, \epsilon)$-quasigeodesic has Hausdorff distance at most $C_{k, \epsilon}$ from the unique geodesic between its endpoints.
\end{proposition}

\subsection{Quasigeodesic flows}
A flow $\rF$ on a manifold $M$ is a continuous $\bR$-action on $M$, i.e. a map
\[ \rF: \bR \times M \to M \]
such that $\rF_t(\rF_s(p)) = \rF_{t+s}(p)$ for all $t, s \in \bR$ and $p \in M$. A flow is called \emph{nonsingular} if it has no global fixed points.

We denote the universal cover of a manifold $M$ by $\widetilde{M}$. A flow $\rF$ on $M$ lifts canonically to a flow $\widetilde{\rF}$ on $\widetilde{M}$.

\begin{definition}
A nonsingular flow $\rF$ on a manifold $M$ is called \emph{quasigeodesic} if each flowline in $\widetilde{\rF}$ is a quasigeodesic. It is called \emph{uniformly quasigeodesic} if there are universal constants $k, \epsilon$ such that each flowline in $\widetilde{\rF}$ is a $(k, \epsilon)$-quasigeodesic.
\end{definition}

It turns out that there is no need to distinguish between quasigeodesic and uniformly quasigeodesic flows in our context:

\begin{lemma}[Calegari, \cite{Calegari}, Lemma 3.10] \label{lemma:UniformlyQuasigeodesic} 
Let $\rF$ be a quasigeodesic flow on a closed hyperbolic 3-manifold $M$. Then $\rF$ is uniformly quasigeodesic for some constants $k, \epsilon$.
\end{lemma} 

The quasigeodesic property for flows is stable under $C^0$ perturbations because of this lemma and the fact that quasigeodesity of flowlines is a local condition. This is in contrast with the geodesible property.

A smooth reparametrization of a flow on a compact space is bilipschitz when restricted to each flowline and hence preserves quasigeodesity. Therefore we are mostly interested in the corresponding foliation by flowlines. We will use the same symbol $\rF$ to refer to both a flow and its corresponding foliation and write $l \in \rF$ to mean that $l$ is a flowline/leaf of $\rF$.

\begin{example}[Zeghib, \cite{Zeghib}]
Let $M$ be a closed surface bundle over the circle. Then any flow that is transverse to the foliation by surfaces is quasigeodesic.
\end{example}

\begin{example}[Fenley-Mosher, \cite{Fenley_Mosher}]
Any closed hyperbolic 3-manifold with nontrivial second homology admits a quasigeodesic flow.
\end{example}

Let $M$ be a closed hyperbolic 3-manifold with a quasigeodesic flow $\rF$. By Lemma~\ref{lemma:UniformlyQuasigeodesic} and Proposition~\ref{prop:QuasisInH} the maps
\[ e^\pm: \widetilde{M} \to \partial \bH^3 \]
that send each point to the positive/negative end of the corresponding flowline are continuous. In fact the existence of these maps characterizes quasigeodesic flows.

\begin{theorem}[Calegari, Fenley-Mosher]
Let $M$ by a closed hyperbolic 3-manifold with a flow $\rF$. Then $\rF$ is quasigeodesic if and only if the maps $e^\pm$ are well defined and continuous and $e^+(p) \neq e^-(p)$ for all $p \in \widetilde{M}$.
\end{theorem}

The ``if'' direction is \cite{Fenley_Mosher}, Theorem B and the ``only if'' direction is \cite{Calegari}, Lemma 4.3.

\section{Main results}\label{sec:Outline}
In this section we state and prove our main theorems modulo certain technical details which are relegated to Sections~\ref{sec:Ends}-\ref{sec:UCandEC}.

\subsection{The orbit space}
Fix a closed hyperbolic 3-manifold $M$ with a quasigeodesic flow $\rF$. Lift to a flow $\widetilde{\rF}$ on the universal cover $\widetilde{M} \simeq \bH^3$. The \emph{orbit space} $\orb$ is the set of flowlines in $\widetilde{\rF}$ with the quotient topology obtained from $\widetilde{M}$ by collapsing each flowline to a point. The action of $\pi_1(M)$ on $\bH^3$ preserves the foliation $\widetilde{\rF}$, so it descends to an action on $\orb$. The orbit space is homeomorphic to the plane (\cite{Calegari}, Theorem 3.12) and the two \emph{endpoint maps}
\[ e^\pm: \orb \to \partial \bH^3 \]
that send each flowline to its positive/negative endpoint in $\partial \bH^3$ are continuous (\cite{Calegari}, Lemma 4.3).

The maps $e^\pm$ are between manifolds of the same dimension. However, they are far from being locally injective. In fact:

\begin{proposition}[Calegari, \cite{Calegari}, Lemma 4.8]\label{prop:PreimageUnbounded} 
Let $\rF$ be a quasigeodesic flow in a closed hyperbolic 3-manifold $M$ and let $p \in \partial \bH^3$. Then each component of $(e^\pm)^{-1}(p)$ is unbounded.
\end{proposition}

The point preimages of the maps $e^\pm$ give rise to interesting structure on $\orb$.

\begin{definition}
A closed, connected, unbounded set in the plane is called an \emph{unbounded continuum}.

An \emph{unbounded decomposition} of the plane $P$ is a collection of unbounded continua covering $P$ such that distinct decomposition elements are disjoint.

A \emph{generalized unbounded decomposition} of the plane $P$ is a collection of unbounded continua covering $P$ such that distinct decomposition elements intersect in a compact set.
\end{definition}

The sets
\[ \cD^\pm := \{\text{components of } (e^\pm)^{-1}(p) | p \in \partial \bH^3 \}, \]
are unbounded decompositions of the orbit space by Proposition~\ref{prop:PreimageUnbounded}. The set
\[ \cD = \cD^+ \cup \cD^- \]
is a generalized unbounded decomposition by the following lemma.

\begin{lemma} \label{lem:CompactSetOfFlowlines}
Let $A$ be a compact subset of the space $\partial \bH^3 \times \partial \bH^3 \setminus \Delta$ where $\Delta $ is the diagonal. Then the preimage of $A$ under the map
\[ e^+ \times e^-: \orb \to \partial \bH^3 \times \partial \bH^3 \]
is compact.
\end{lemma}
\begin{proof}
There is a compact set $K \subset \bH^3$ that intersects every geodesic whose endpoints are a pair in $A$. Recall that there is a uniform constant $C$ such that each flowline has Hausdorff distance at most $C$ from the geodesic with the same endpoints (Lemma~\ref{lemma:UniformlyQuasigeodesic} and Proposition~\ref{prop:QuasisInH}), so the flowlines whose ends are pairs in $A$ all intersect the $C$-neighborhood, $L$, of $K$. The preimage of $A$ is contained in the projection of $L$ to the orbit space and hence compact.
\end{proof}

In the sequel we will concentrate on  $\cD^+$ for notational simplicity. Everything we say works, \emph{mutatis mutandis}, for $\cD^-$ and $\cD$.

\subsection{Universal circles for quasigeodesic flows}
To each decomposition element $K \in \cD^+$ we associate the set of \emph{ends} $\cE(K)$ (see Section~\ref{sec:Ends}, especially Definition~\ref{def:Ends} for a discussion of ends and a precise definition). We lump together the ends of all of the positive decomposition elements in one set
\[ \cE^+ := \bigcup_{K \in \cD^+} \cE(K). \]
The action of $\pi_1(M)$ on $\orb$ preserves the decomposition $\cD^+$ so it induces an action on $\cE^+$.

In Section~\ref{sec:CO} we will show that the sets of ends $\cE^+$ comes with a natural circular order that is preserved by the action of $\pi_1(M)$. This has the nice property that we can tell whether $K \in \cD^+$ separates $L, M \in \cD^+$ in $\orb$ by whether $\cE(K)$ separates $\cE(L)$ and $\cE(M)$ in the circular order (Proposition~\ref{prop:DetectSeparation}). In addition, after taking the order completion and collapsing some intervals (Section~\ref{sec:UniversalCircles}) we can form the \emph{positive universal circle} $S_u^+$, which is homeomorphic to $S^1$. There is an order-preserving map 
\[ \phi: \cE^+ \to S_u^+, \]
with dense image and a natural \emph{faithful} action of $\pi_1(M)$ on $S_u^+$ that is equivariant with respect to $\phi$.

In Section~\ref{sec:EndCompactification} we show that there is a natural \emph{end compactification} $\overline{\orb}^+$ of $\orb$ with respect to the decomposition $\cD^+$. This is homeomorphic to a closed disc with boundary $\partial \overline{\orb}^+ = S^+_u$. The action of $\pi_1(M)$ on $\orb$ extends to an action on $\overline{\orb}^+$ that restricts to the universal circle action on $S^+_u$.

Replacing $\cD^+$ with $\cD^-$ or $\cD$ yields the compactifications $\overline{\orb}^-$ and $\overline{\orb}$.

\subsection{Closed orbits}\label{sec:ClosedOrbits}
A fixed point $p \in S^1$ of $g \in \text{Homeo}^+(S^1)$ is \emph{attracting} if there is a neighborhood $U$ of $p$ such that for every open interval $I \subset U$ containing $p$ we have $\overline{g(I)} \subset I$. It is \emph{repelling} if instead we have $\overline{I} \subset g(I)$. A fixed point is \emph{indifferent} if it is neither attracting nor repelling. If $g \in \text{Homeo}^+(S^1)$ has two fixed points then they are either both indifferent or an attracting-repelling pair.

\begin{theorem}\label{thm:main}
Let $M$ be a closed hyperbolic 3-manifold with a quasigeodesic flow $\rF$. Fix an element $g \in \pi_1(M)$ and suppose that $g$ acts on $S^+_u$ (or $S^-_u$, or $S_u$) with either more than two fixed points, two indifferent fixed points, or no fixed points. Then $\rF$ has a closed orbit in the free homotopy class of $g$.
\end{theorem}
\begin{proof}
We will prove this for $S^+_u$. The argument is identical for $S^-_u$ and $S_u$.

Note that a closed orbit in the free homotopy class of $g$ is a point in $\orb$ fixed by $g$.

The action of $g$ on $\partial \bH^3$ has an attracting fixed point $a_g$ and a repelling fixed point $r_g$. Suppose there is a flowline $\gamma \in \widetilde{\rF}$ with an endpoint at either $a_g$ or $r_g$. Then $g$ has a fixed point in $\orb$. Indeed, let $\gamma$ have an endpoint at $r_g$ (replace $g$ by $g^{-1}$ if it has an endpoint at $a_g$). Then the endpoints of $g^n(\gamma)$ for $n$ positive approach $(a_g, r_g)$ so the forward orbit of $\gamma$ is bounded in $\orb$ by Lemma~\ref{lem:CompactSetOfFlowlines}. The Brouwer plane translation theorem implies that $g$ must have a fixed point in $\orb$ (see \cite{Franks}).

Now suppose that $g$ has at least three fixed points $x, y, z \in S^+_u$. Choose decomposition elements $K, L, M \in \cD^+$ that have ends in each of the respective oriented intervals $(x, y)$, $(y, z)$, and $(z, x)$ in $S^+_u$. Let $A$ be a compact set intersecting $K$, $L$, and $M$ and set $B = A \cup K \cup L \cup M$. The image of $B$ under $e^+$ is compact since $e^+(B) = e^+(A)$ and $A$ is compact. Suppose that $e^+(B)$ does not contain $a_g$ or $r_g$. Then for sufficiently large $n$, $e^+(B) \cap g^n(e^+(B)) = \emptyset$ and so $B \cap g^n(B) = \emptyset$. But this is impossible since the ends of $B$ and $g^n(B)$ link in $S^+_u$ and hence $B$ and $g^n(B)$ must intersect by Proposition~\ref{prop:DetectSeparation}. Therefore, either $a_g$ or $r_g$ are in $e^+(B)$ and the preceeding discussion yields a closed orbit. See Figure~\ref{fig:3fixedpoints} for an illustration of this argument.

\begin{figure}[h]
	\labellist
	\small\hair 2pt
	\pinlabel $\orb$ [bl] at 12 315
	\pinlabel $x$ [b] at 74 326
	\pinlabel $z$ [tl] at 142 219
	\pinlabel $y$ [tr] at 11 211
	\pinlabel $A$ [tl] at 92 240
	\pinlabel $K$ [bl] at 32 290
	\pinlabel $L$ [tr] at 81 197
	\pinlabel $M$ [r] at 124 296

	\pinlabel $e^+$ [b] at 184 284

	\pinlabel $\partial \bH^3$ [br] at 231 315
	\pinlabel $a_g$ [b] at 292 326
	\pinlabel $r_g$ [t] at 292 177
	\pinlabel $e^+(B)$ [l] at 267 282

	\pinlabel $x$ [b] at 74 152
	\pinlabel $z$ [tl] at 142 41
	\pinlabel $y$ [tr] at 11 34
	\pinlabel $g^n(A)$ [bl] at 82 115
	\pinlabel $g^n(K)$ [tl] at 61 142
	\pinlabel $g^n(L)$ [r] at 36 72
	\pinlabel $g^n(M)$ [tr] at 131 72

	\pinlabel $e^+$ [b] at 184 114

	\pinlabel $a_g$ [b] at 292 153
	\pinlabel $r_g$ [t] at 292 1
	\pinlabel $e^+(g^n(B))$ [tl] at 285 132
	\endlabellist
	\centering
		\includegraphics{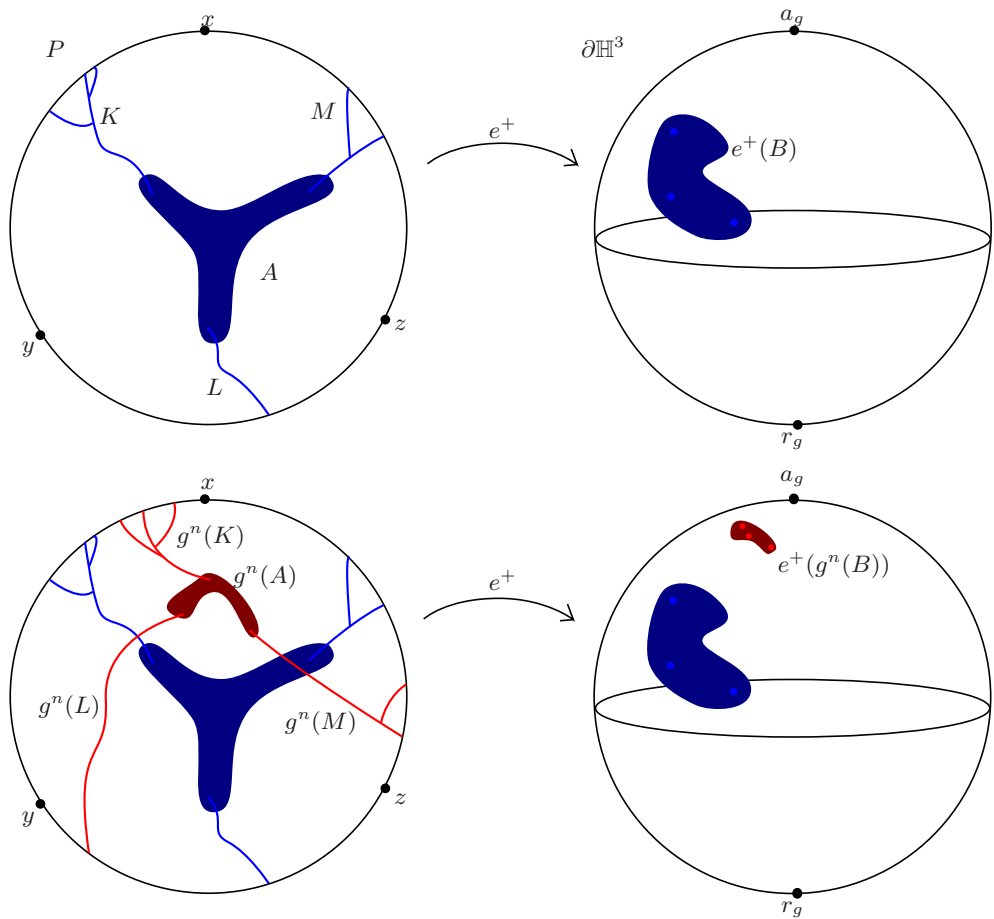}
	
	\caption{The argument for closed orbits when $g$ has three fixed points on $S^+_u$.} \label{fig:3fixedpoints}
\end{figure}

The argument for two indifferent fixed points is similar. See Figure~\ref{fig:2fixedpoints}.

\begin{figure}[h]
	\labellist
	\small\hair 2pt
	\pinlabel $\orb$ at 12 146

	\pinlabel $g^n(A)$ [t] at 86 50
	\pinlabel $g^n(K)$ [bl] at 31 109
	\pinlabel $g^n(L)$ [b] at 136 59
	\endlabellist
	\centering
		\includegraphics{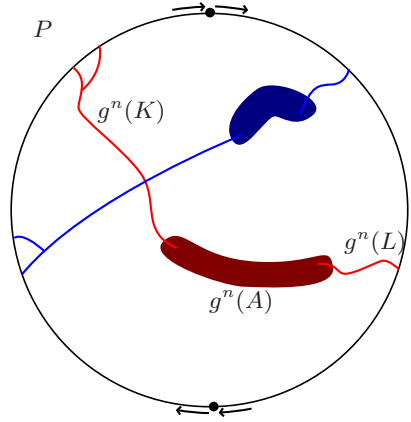}
	\caption{The corresponding argument for two indifferent fixed points.} \label{fig:2fixedpoints}
\end{figure}

Suppose $g$ has no fixed points in $S^+_u$. By the Compactification Theorem the action of $g$ on $S^+_u$ is identical to the action of $g$ on $\partial \overline{\orb}$, so by the Brouwer fixed point theorem $g$ fixes a point in $\orb$.
\end{proof}

A group acting on the circle is a rotationless M\"obius-like group if and only if each element acts with either one fixed point or two fixed points in an attracting-repelling pair. The following is therefore an immediate corollary of Theorem~\ref{thm:main}.

\begin{mob_theorem}
Let $\rF$ be a quasigeodesic flow on a closed hyperbolic 3-manifold $M$. Suppose that $\pi_1(M)$ acting on the universal circle $S^+_u$ (or $S^-_u$ or $S_u$) is not a rotationless M\"obius-like group. Then $\rF$ has a closed orbit.
\end{mob_theorem}

On the other hand:
\begin{conj_theorem}
Let $\rF$ be a quasigeodesic flow on a closed hyperbolic 3-manifold $M$. Then the action of $\pi_1(M)$ on $S^+_u$ (or $S^-_u$ or $S_u$) is not conjugate into $PSL(2, \bR)$.
\end{conj_theorem}
\begin{proof}
Suppose that the action of $\pi_1(M)$ on $S^+_u$ were conjugate into $PSL(2, \bR)$. It cannot be a discrete subgroup of $PSL(2, \bR)$ since then $\pi_1(M)$ would be isomorphic to a surface group. The closure is a Lie group by the Cartan lemma, so it must be all of $PSL(2, \bR)$ since all proper Lie subgroups of $PSL(2, \bR)$ are solvable. But then some element would act without fixed points since every fixed point-free element of $PSL(2, \bR)$ has a neighborhood containing only fixed point-free elements. This contradicts the M\"obius-like Theorem.
\end{proof}

\section{Ends}\label{sec:Ends}
The next few sections are mainly technical, and dedicated to proving the necessary structure theorems to justify the arguments in Section~\ref{sec:Outline}. The decomposition elements that arise from the endpoint maps of a quasigeodesic flow can be arbitrarily complicated closed subsets of the plane so some care is needed.

\begin{definition} \label{def:Ends}
Let $P$ be homeomorphic to the plane and let $K \subset P$. An \emph{end} of $K$ is a map
\[ \kappa: \{\text{bounded subsets of } P\} \to \{\text{nonempty subsets of } K\} \]
such that
\begin{enumerate}
	\item for each bounded set $D \subset P$, $\kappa(D)$ is a connected component of $K \setminus D$, and
	\item if $D' \supset D$ are bounded subsets of $P$ then $\kappa(D') \subset \kappa(D)$.
\end{enumerate}

We write $\cE(K)$ for the set of ends of $K$. If $\cD$ is a collection of subsets of the plane we write $\cE(\cD) := \cup_{K \in \cD} \cE(K)$.
\end{definition}

To specify an end it is enough to keep track its values on any bounded exhaustion of $P$. Indeed, fix such a bounded exhaustion $(D_i)_{i=0}^{\infty}$. For any bounded $D \subset P$ there is some $i$ such that $D_i \supset D$. Then $\kappa(D)$ is the component of $K \setminus D$ containing $\kappa(D_i)$.

Conversely, we can use this to explicitly specify an end. Let $(D_i)$ be a bounded exhaustion of $P$ and suppose there is a sequence $(K_i)$ such that each $K_i$ is a component of $K \setminus D_i$ and $K_{i+1} \subset K_i$ for all $i$. Then there is a unique end $\kappa \subset \cE(K)$ with $\kappa(D_i) = K_i$ for all $i$.

\begin{lemma}
Let $K \subset P$ be an unbounded continuum in the plane and let $D$ be a bounded set that intersects $K$. Then some component of $K \setminus D$ is unbounded.
\end{lemma}
\begin{proof}
We work in the one point compactification of the plane $\hat{P} \simeq S^2$ and replace $K$ by $\hat{K} = K \cup \{\infty\}$.  Let $\{U_i\}_{i=1}^\infty$ be a sequence of connected open neighborhoods of $\hat{K}$ with intersection $\cap U_i = \hat{K}$ and let $D'$ be a closed disc away from $\infty$ with $\overline{D} \subset int(D')$. Fix a point $p \in \hat{K} \cap D'$ other than $\infty$. For each $i$ we can choose an arc $\gamma_i:[0,1] \to U_i$ from $p$ to $\infty$. Take $\gamma'_i = \gamma_i([t_i, 1])$ where $t_i$ is the last intersection of $\gamma_i$ with $D'$. Let $A$ be the Hausdorff limit of a convergent subsequence of the $\gamma'_i$ and note that $A$ is a compact connected subset of $\hat{K} \setminus D$ containing $\infty$.

Now let $q$ be any point in $A$ other than $\infty$ and find a minimal continuum $A' \subset A$ containing $\{q \cup \infty\}$. Suppose that $A' \setminus \infty$ is disconnected, i.e. $A' = B \cup C$ separated. If $q \in B$ then $B \cup \{\infty\}$ is a continuum containing $\{q \cup \infty\}$ contradicting the minimality of $A'$. So $A' \setminus \{\infty\}$ is contained in an unbounded subset of $K \setminus D$.
\end{proof}

\begin{proposition}
Let $K \subset P$ be an unbounded continuum in the plane. Then $K$ has at least one end.
\end{proposition}
\begin{proof}
Let $(D_i)_{i=1}^\infty$ be an exhaustion of the plane by nested bounded open sets. For each $i$ let $K_i$ be an unbounded component of $K_{i-1} \setminus D_i$, which exists by the preceeding lemma. Then there is an end $\kappa$ with $\kappa(D_i) = K_i$ for all $i$.
\end{proof}

If $A \in P$ is unbounded but not closed then $A$ does not necessarily have ends. For example:
\begin{example}\label{ex:NoEnds}
Let $A$ be the set in $\bR^2$ consisting of the line segment $A_0 = [0, 1] \times \{0\}$ together with the segments $\{1/n\} \times [0, n]$ for all integers $n \geq 1$. After removing $A_0$ there are no unbounded pieces left so $\cE(A) = \emptyset$. Also note that we can thicken each of the segments to an open neighborhood to construct an example of an open unbounded set with no ends.
\end{example}

On the other hand, if $A \subset B$ and $\cE(A) \neq \emptyset$ then $\cE(B) \neq \emptyset$. In particular, any set containing an unbounded continuum has an end. The following lemma provides a converse of this for open sets.

\begin{lemma}\label{lem:RayApproachingEnd}
Let $U \subset P$ be an open, connected set with at least one end. Then there is an embedded ray $\gamma \subset U$ that is proper in $P$.
\end{lemma}
\begin{proof}
If $W \subset S^2$ is open, a point $p \in \partial W$ is arcwise accessible from $W$ iff it is accessible by a connected closed set (see \cite{Wilder} for a definition of arcwise accessibility and \cite{Wilder}, Thm IV.5.1 for this statement). So it suffices to show that there is an unbounded continuum $K \subset U$ since then the closure of $K$ in the one point compactification $\hat{P} \simeq S^2$ is just $K \cup \{\infty\}$.

Let $\mu$ be an end of $U$ and fix an exhaustion of the plane by nested compact sets $(D_i)_{i=1}^\infty$. For each $i$ choose a point $x_i \in \mu(D_i)$ and a connected closed set $K_i$ from $x_{i}$ to $x_{i+1}$ contained in $\mu(D_i)$. The union $K = \cup K_i$ is clearly connected and unbounded. It is closed since only finitely many $K_i$ intersect any bounded set in the plane.
\end{proof}

We can characterize the ends of the components complementary to an unbounded continuum.

\begin{lemma}\label{lem:ComplementsOneEnded}
Let $K \subset P$ be an unbounded continuum and let $U$ be an unbounded connected component of $P \setminus K$. Then $U$ has at most one end.
\end{lemma}
\begin{proof}
Suppose on the contrary that $U$ has two ends. Then there is a compact set $D$ and distinct connected components $U_1, U_2 \subset U \setminus D$ such that each $U_i$ has at least one end. Then by Lemma~\ref{lem:RayApproachingEnd} there are proper embedded rays $\gamma_i \in U_i$. Connect these rays with an arc $\gamma'$ lying in $U$ and let $\gamma = \gamma_1 \cup \gamma' \cup \gamma_2$. Then $\gamma$ is an properly embedded curve and by the Jordan curve theorem there is a homeomorphism of $P$ taking $\gamma$ to the $x$-axis. Now $K$ is contained in, say, the lower half plane so the entire upper half plane is contained in $U$. But there is an arc in the upper half plane connecting $\gamma_1$ and $\gamma_2$ so $U_1$ and $U_2$ cannot be distinct.
\end{proof}

\section{Circular orders}
\begin{definition}
A \emph{circular ordered set} is a set $S$ with at least three elements together with a map 
\[ \langle \cdot , \cdot , \cdot \rangle: \{\text{triples of distinct elements of } S\} \to \pm 1 \]
such that for distinct $x, y, z \in S$
\begin{itemize}
	\item (antisymmetry condition:) $\langle x, y, z \rangle  = (-1)^{sgn(\tau)} \langle \tau(x), \tau(y), \tau(z) \rangle $ for $\tau$ a permutation of $\{x, y, z\}$, and
	\item (cocycle condition:) if $\langle x, y, z \rangle  = +1$ and $\langle x, z, w \rangle  = +1$ then $\langle x, y, w \rangle  = +1$ and $\langle y, z, w \rangle = +1$.
\end{itemize}
\end{definition}

The triple $x,y,z$ is said to be positively ordered if $\langle x,y,z \rangle  = +1$ and negatively ordered otherwise. A tuple $x_1, x_2, ..., x_n$ is said to be positively ordered if $x_i, x_{i+1}, x_{i+2}$ is positively ordered for each $i$ with subscripts taken mod $n+1$.

The \emph{order topology} on a circularly ordered set is the topology with basis the \emph{open intervals}
\[ (x, y) = \{t \in S | \langle x, t, y \rangle  = +1\}.\]
We define the \emph{closed intervals}
\[ [x, y] = (x,y) \cup \{x, y\} \]
and note that $[x, y]$ is the complement of $(y, x)$, which we denote by $(y, x)^c$ in the sequel.

\subsection{The order completion}
\begin{definition}
A circularly ordered set $S$ is \emph{order complete} if for any sequence $I_1 \supset I_2 \supset ...$ of nested closed intervals we have $\bigcap_n I_n \neq \emptyset$.
\end{definition}

Every circularly ordered set has a canonical \emph{order completion} $\overline{S} \supset S$ constructed as follows.

\begin{construction}
Let $S$ be a circularly ordered set. An \emph{admissible sequence} in $S$ is an infinite sequences of closed intervals $(I_i)_{i = i_0}^\infty$ such that $I_i \supseteq I_{i+1}$ for all $i$ and $\bigcap_{i = n}^\infty I_i = \emptyset$. Let $S'$ be the set of equivalence classes of admissible sequences under the following relation.

We define $(I_i) \sim (J_j)$ if for each $n > 0$ there exists $k > 0$ such that $I_k \subset J_n$. This is indeed an equivalence relation. To see that it is symmetric, suppose $(I_i) \sim (J_j)$ and let $n > 0$. Then if $k$ is large enough so that $J_k$ does not contain the endpoints of $I_n$ we must have $J_k \subset I_n$ since $I_{k'} \subset J_k$ for large enough $k'$. Hence $(J_j) \sim (I_i)$. Reflexivity and transitivity are obvious. Note as well that a sequence is equivalent to any of its subsequences.

The order completion is
\[ \overline{S} = S \cup S'. \]
To define the circular order on $\overline{S}$ we can represent each point $x \in S \subset \overline{S}$ by the constant infinite sequence $([x, x])$. If $(I_i), (J_j), (K_k) \in \overline{S}$ are distinct we can choose $n$ large enough such that the intervals $I_n, J_n, K_n$ are disjoint. Then choose $x \in I_n$, $y \in J_n$, and $z \in K_n$ and define $\langle (I_i), (J_j), (K_k) \rangle  = \langle x, y, z \rangle $. It is easy to see that this is well-defined.
\end{construction}

\begin{proposition}
If $S$ is a circularly ordered set then the order completion $\overline{S}$ is in fact order complete. Further, $S$ is dense in $\overline{S}$ with the order topology.
\end{proposition}
\begin{proof}
Suppose that $[(I_i^n), (J_j^n)]_{n=1}^\infty$ is a nested sequence of closed intervals in $\overline{S}$. For each $n$ choose $i_n$ so that $I^{n-1}_{i_{n-1}}$, $I^n_{i_n}$, and $I^{n+1}_{i_{n+1}}$, are disjoint and choose $x_n \in I^n_{i_n}$. Choose $y_n \in J^n_{j_n}$ similarly. Then either $\bigcap_n [x_n, y_n] \neq \emptyset$ or $[x_n, y_n]$ is an admissible sequence. So $\overline{S}$ is order complete.

The second statement is obvious.
\end{proof}

\subsection{Universal circles for circularly ordered sets} \label{sec:UCforCO}
A \emph{gap} in a circularly ordered set $S$ is an ordered pair $x, y \in S$ such that $(x, y)$ is empty.

\begin{proposition} \label{prop:2ndCountable}
Let $S$ be a separable circularly ordered set with countably many gaps. Then $S$ is 2nd countable.
\end{proposition}
\begin{proof}
Let $S'$ be a countable dense subset of $S$ and let $S''$ to be the set of endpoints of gaps. Let $\cU$ be the collection of open intervals with endpoints in $S' \cup S''$.

Suppose that $U \subset S$ is open and let $x \in U$. Then $x \in (a, b) \subset U$ for some $a, b$. If $(a, x) \neq \emptyset$ then we can find $a' \in S' \cap (a, x)$; otherwise $a \in S''$ and set $a' = a$. Find $b'$ similarly. Then $(a', b') \in \cU$ and $x \in (a', b') \subset U$. This shows that $\cU$ is a (countable) basis for the order topology on $\cE(\cD)$.
\end{proof}

\begin{proposition}
Let $S$ be a separable circularly ordered set with countably many gaps. Then there is an order-preserving injection
\[ f:S \to S^1. \]
If in addition $S$ is order complete then $f$ may be chosen to be a continuous map with closed image.
\end{proposition}
\begin{proof}
Let
\[ S' = \{s_i\}_{i=1}^\infty\]
be a countable dense subset of $S$ that contains the endpoints of all gaps in $S$. We will start by defining $f$ on $S'$. Send $s_1$ and $s_2$ to two antipodal points in $S^1$. Once we have defined $f$ for $S'_{n-1} = \{s_i\}_{i = 1}^{n-1}$, there are unique $a, b < n$ such that $s_n$ is the only element of $S'_n$ lying in $(s_a, s_b)$. Let $f(s_n)$ be the midpoint of the positively oriented interval from $f(s_a)$ to $f(s_b)$.

Now suppose $x \in S \setminus S'$. Note that $S$ is 2nd countable by Proposition~\ref{prop:2ndCountable}, so we can find a sequence $x_i$ that approaches $x$, in a counterclockwise direction (that is, $x_{i+1} \subset (x_i, x)$ for all $i$). Define
\[ f(x) = \lim_{i \to \infty} f(x_i). \] 
The clockwise limit is identical since $S'$ contains the endpoints of gaps so it is clear that this is well-defined and that $f$ is order-preserving and injective.

Note that if $(a, b)$ is a maximal open interval in $S^1 \setminus f(S)$ then both $a$ and $b$ must be contained in $f(S)$. Indeed, otherwise we could find $i$ and $j$ such that $s_i$ and $s_j$ are arbitrarily close to $a$ and $b$. Then for $k$ large enough, $f(s_k)$ would be the midpoint of the interval $(f(s_i), f(s_j))$ and hence contained in $(a, b)$.

Now suppose that $S$ is order complete. To show that $f$ is continuous it is enough to show that $f(S)$ is closed, since if $(a, b) \subset S^1$ is any open interval we can find $x, y \in S$ such that $(f(x), f(y)) \cap f(S) = (a, b) \cap S$. Then
\[ f^{-1}((a, b)) = f^{-1}((f(x), f(y)) = (x, y). \]

To see that $f(S)$ is closed, let $a \in \overline{f(S)} \subset S^1$. If $a$ is an endpoint of a complementary interval of $f(S)$ then $a$ is in $S$. Otherwise, $a$ is approached by $f(S)$ from both sides. That is, there are sequences $(x_i)$ and $(y_i)$ in $S$ such that $a \in (f(x_i), f(y_i))$ and $(x_{i+1}, y_{i+1}) \subset (x_{i}, y_{i})$ for all $i$. But $S$ is order complete, so there must be an $x \in \bigcap [x_i, y_i]$ and $f(x) = a$.
\end{proof}

\begin{construction}\label{con:UCforCO}
Let $S$ be an uncountable and separable circularly ordered set with countably many gaps. Then the order completion $\overline{S}$ is homeomorphic to an uncountable closed subset of the circle. By the Cantor-Bendixson theorem (see \cite{Kechris}, Theorem 6.4) $\widetilde{S} = T \cup U$ where $T$ is a closed perfect set and $U$ is countable. We collapse the closures of complementary intervals to $T$ to construct the \emph{universal circle} $\widehat{\overline{S}}$, which is homeomorphic to the circle. Any automorphism of a circularly ordered set $S$ induces a homeomorphism on $\widehat{\overline{S}}$.
\end{construction}

\section{Unbounded decompositions}
Throughout this section, $P$ will be a topological space that is homeomorphic to the plane. A collection $\cD$ of disjoint subsets of the plane is \emph{mutually nonseparating} if no one set separates the others. Equivalently, for each $K \in \cD$ there is a single component $U$ of $K^c$ such that each $L \in \cD \setminus \{K\}$ is contained in $U$.

An \emph{$n$-ad} in the plane is a set of $n$ mutually nonseparating unbounded continua. We will show that any $n$-ad has a natural circular order induced by an orientation of the plane.

Suppose $\cD$ is generalized unbounded decomposition of the plane $P$ and $\cE = \cE(\cD)$ is the corresponding set of ends. If $\kappa, \lambda, \mu \in \cE$ then we can find a bounded disc $D$ such that $\kappa(D)$, $\lambda(D)$, and $\mu(D)$ form a triad. Such a disc is said to \emph{distinguish} the ends $\kappa$, $\lambda$, and $\mu$.

The circular order on $n$-ads will induce a circular order on $\cE$.

\subsection{Topological Background}
Let's collect a few definitions and observations that will be useful in the next section. We will use some classical facts from planar point-set topology; we give references in the text and statements in the footnotes.

\begin{lemma}\label{lem:ComplementaryComponents}
Let $K_1, K_2, ..., K_n \subset P$ be an n-ad in the plane. There is exactly one component $C(K_1, ..., K_n)$ of $(\bigcup_i K_i)^c$ that limits on all of the $K_i$. Every other component limits on only one of the $K_i$.
\end{lemma}
\begin{proof}
Set
\[ C = C(K_1, ..., K_n) = \bigcap_i U_i. \]
For each $i$ let $U_i$ be the component of $K_i^c$ that contains the $K_j$ for $j \neq i$. Set $K'_i := P \setminus U_i$ for each $i$. Each $K'_i$ is connected by \cite{Wilder}, Thm. I.9.11. \footnote{If $C$ is a connected subset of a connected space $M$ and $A$ is a component of $M \setminus C$, then $M \setminus A$ is connected.}

Note that $C = (\bigcup_i K'_i)^c$.

The set $C$ is connected since each $K'_i$ is nonseparating and the union of finitely many disjoint nonseparating sets in the plane is nonseparating (\cite{Wilder}, Thm. II.5.28a). It is a maximal connected set since any $x \in P \setminus \bigcup_{i} (K_i)$ that is not in $C$ is separated from $C$ by some $K_i$. Thus $C$ is a component of $(\bigcup_i (K_i))^c$. It is clear that $\overline{C}$ intersects every $K_i$.

Every other component of $(\bigcup_i (K_i))^c$ is a component of $K_i^c$ for some $i$.
\end{proof}

\begin{lemma}
Let $K, L, M \subset P$ be a triad. Then $L \subset C(K, M)$.
\end{lemma}
\begin{proof}
Note that $C(K, L, M)$ limits on $K$ and $M$, so it must be contained in $C(K, M)$. But $C(K, L, M)$ also limits on $L$ so $L$ must also be in $C(K, M)$.
\end{proof}

Let $K, M \subset P$ be closed sets in the plane. An \emph{arc from $K$ to $M$} is an embedded oriented arc with initial point in $K$ and terminal point in $M$ whose interior $\mathring{\gamma}$ is disjoint from $K \cup M$. Note that $C(K, M)$ is homeomorphic to the disc (by the uniformization theorem, since the compact region bounded by any simple closed curve cannot intersect $K$ or $M$) and $\mathring{\gamma} \subset C(K, M)$. By the Jordan curve theorem, $\mathring{\gamma}$ separates $C(K,M)$ into two discs. We define $C^+(\gamma; K, M)$ to be the component of $C(K, M) \setminus \gamma$ on the positive side of $\gamma$ and $C^-(\gamma; K, M)$ to be the component on the negative side.

If the sets $K$ and $M$ are implicit then we use the abbreviation $C^\pm(\gamma) = C^\pm(\gamma; K, M)$.

Note that $C^\pm(\gamma)$ are discs for the same reason that $C(K, M)$ is.

\begin{lemma} \label{lem:SeparatingArc}
Let $K, M \subset P$ be disjoint unbounded continua in the plane and let $\gamma$ be an arc from $K$ to $M$. Then $C^\pm(\gamma; K, M)$ are both unbounded.
\end{lemma}
\begin{proof}
Suppose that one of these, say $C^+(\gamma)$, is bounded. Then $\partial C^+(\gamma) = K' \cup \gamma \cup M'$ where $K'$ and $M'$ are bounded subsets of $K$ and $M$. Choose a point $p \in C$. Now neither $K'$ nor $\gamma$ separate $p$ from $\infty$ in the one point compactification $\hat{P} \simeq S^2$ and $K' \cap \gamma$ is connected. So $K' \cup \gamma$ does not separate $p$ from $\infty$ (\cite{Wilder}, Thm II.5.29). \footnote{If $x$ and $y$ are points in $S^2$ which are not separated by either of the closed sets $A$ and $B$, and $A \cap B$ is connected, then $x$ and $y$ are not separated by $A \cup B$.} Similarly, neither $K' \cup \gamma$ nor $L'$ separate $p$ from $\infty$ and $(K' \cup \gamma) \cap L'$ is connected so $(K' \cup \gamma) \cup L'$ does not separate $p$ from $\infty$ contradicting our assumption.
\end{proof} 

\begin{lemma}\label{lem:ArcSides}
Let $K, M \subset P$ be disjoint unbounded continua in the plane and let $\gamma, \gamma'$ be disjoint arcs from $K$ to $M$. Then we can relabel $\gamma$ and $\gamma'$ as the ``outer arc'' $\gamma_+$ and ``inner arc'' $\gamma_-$ in such a way that
\begin{enumerate}
	\item $\gamma_+ \subset C^+(\gamma_-)$ and $\gamma_- \subset C^-(\gamma_+)$;
	\item the components of $C(K, M) \setminus (\gamma_+ \cup \gamma_-)$ are 
			\[ C^+(\gamma_+), \]
			\[ C^-(\gamma_+) \cap C^+(\gamma_-), \]
			and
			\[ C^-(\gamma_-); \]
	\item $C^+(\gamma_+) \subset C^+(\gamma_-)$ and $C^-(\gamma_-) \subset C^-(\gamma_+)$; and
	\item $C^-(\gamma_+) \cap C^+(\gamma_-)$ has no ends.
\end{enumerate}
\end{lemma}
\begin{proof}
Suppose $\gamma \subset C^+(\gamma')$. Then label $\gamma_+ := \gamma$ and $\gamma_- := \gamma'$. We need to show that
\[ \gamma_- \subset C^-(\gamma_+). \]
If on the contrary $\gamma_- \subset C^+(\gamma_+)$ then there is an oriented arc $\lambda$ from $\gamma_-$ to $\gamma_+$ whose interior $\mathring{\lambda}$ is on the positive side of both $\gamma_+$ and $\gamma_-$. Let $\lambda'$ be an arc from $\gamma_+ \cap K$ to $\gamma_- \cap K$ that misses $L$. Then the simple closed curve $c$ consisting of $\lambda$, the segment of $\gamma_+$ from $\lambda$ to $K$ with orientation reversed, $\lambda'$, and the segment of $\gamma_-$ from $K$ to $\lambda$ separates $L \cap \gamma_+$ from $L \cap \gamma_-$, which is impossible.

If $\gamma' \subset C^+(\gamma)$ then a similar argument works after setting $\gamma_+ := \gamma'$ and $\gamma_- := \gamma$.

This completes the proof of (1). Statements (2) and (3) follow.

Suppose that $C^-(\gamma_+) \cap C^+(\gamma_-)$ has an end. Then by Lemma~\ref{lem:RayApproachingEnd} there is a properly embedded ray $\lambda:[0, \infty) \to C^-(\gamma_+) \cap C^+(\gamma_-)$. Let $\kappa$ be an arc from $\gamma_+$ to $\gamma_-$ that intersects $\lambda$ only at $\lambda(0)$. Either $K$ or $L$ is on the same side of $\kappa$ as $\lambda$, so suppose without loss of generality that $K$ is. Then we can add an arc $\lambda'$ from $\lambda(0)$ to $L$ whose interior is on the opposite side of $\kappa$. Observe that $L \cup \lambda \cup \lambda'$ separates $\gamma_+$ from $\gamma_-$ and hence separates $\gamma_+ \cap K$ from $\gamma_- \cap K$, which is impossible.
\end{proof}

While $C^-(\gamma_+) \cap C^+(\gamma_-)$ in the preceeding lemma has no ends, it is possible that it is unbounded. For example, see Figure~\ref{fig:UnboundedMiddle} where $K$ is upper rectangle with the no-ended open set of Example~\ref{ex:NoEnds} cut out.

\begin{figure}[h]
	\centering
	\labellist
	\small\hair 2pt
	\pinlabel $K$ at 80 131
	\pinlabel $L$ at 80 24
	\pinlabel $\gamma_+$ [r] at 37 75
	\pinlabel $\gamma_-$ [l] at 127 75
	\endlabellist
		\includegraphics{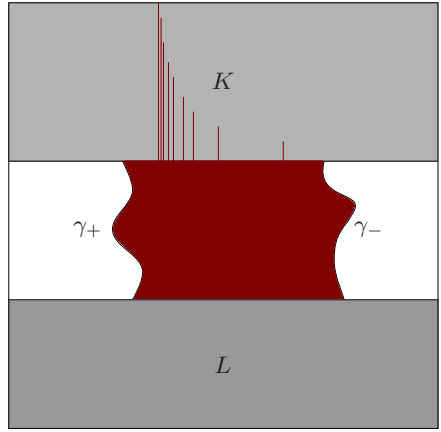}
	\caption{An unbounded middle region.}\label{fig:UnboundedMiddle}
\end{figure}

\subsection{Circular orders on ends} \label{sec:CO}
\begin{definition}
Let $P$ be the plane with a choice of orientation and let $K, L, M$ be a triad in $P$. Choose an arc $\gamma$ from $K$ to $M$. Then we define
\[ \langle K, L, M \rangle  = +1 \]
if $L \subset C^+(\gamma)$
and 
\[ \langle K, L, M \rangle  = -1 \]
if $L \subset C^-(\gamma)$
\end{definition}

\begin{proposition}
The function $\langle \cdot, \cdot, \cdot \rangle $ defines a circular order an n-ad in the plane.
\end{proposition}
\begin{proof}
We will start by showing show that for a triad $K, L, M$ in $P$, $\langle K, L, M \rangle $ does not depend on the choice of arc $\gamma$. Note that if $\gamma_1$ and $\gamma_2$ are arcs from $K$ to $M$ then we can find another arc from $K$ to $M$ disjoint from both of them. So we may assume without loss of generality that $\gamma_1$ and $\gamma_2$ are disjoint. Relabel the arcs according to Lemma~\ref{lem:ArcSides} and suppose that $L \subset C^+(\gamma_+)$. By part (4) of the lemma $L$ is contained in either $C^+(\gamma_+)$ or $C^-(\gamma_-)$, so $\langle K, L, M \rangle $ is well-defined by part (3) of the lemma.

Next we'll show that
\[ \langle K, L, M \rangle  = (-1)^{sgn(\tau)} \langle \tau(K), \tau(L), \tau(M) \rangle , \]
for a permutation $\tau$ of $(K, L, M)$.

It is immediate that $\langle K, L, M \rangle  = -\langle M, L, K \rangle $, so we just need to show that $\langle K, L, M \rangle  = -\langle K, M, L \rangle $. Indeed, assume that $\langle K, L, M \rangle  = +1$, i.e. that $L \subset C^+(\gamma)$ for some arc $\gamma$ from $K$ to $M$. We can find an arc $\gamma'$ from $K$ to $M$ that intersects $L$ and lies on the positive side of $\gamma$. Let $\gamma''$ be the sub-arc of $\gamma'$ that runs from $K$ to $L$. Then it is immediate that $\gamma$ is on the negative side of $\gamma''$ and hence so is $M$. Therefore, $\langle K, M, L \rangle  = -1$ as desired.

It remains to show that the circular order is compatible on quadruples, i.e. to verify the cocycle condition. Let $K, L, M, N$ be a $4$-ad in $P$. Suppose that $\langle K, L, M \rangle  = +1$ and $\langle K, N, M \rangle  = -1$. Choose an arc $\gamma$ from $K$ to $M$ that avoids $L$ and $N$. Then $L \in C^+(\gamma)$ while $N \in C^-(\gamma)$. Let $\gamma'$ be an arc from $L$ to $N$, which we can choose to intersect $\gamma$ only once and transversely. Then the initial segment of $\gamma$ is on the negative side of $\gamma'$ hence so is $K$. Therefore we have $\langle L, K, N \rangle  = -1$ and $\langle K, L, N \rangle  = +1$ as desired and similarly $\langle L, M, N \rangle = +1$.
\end{proof}

A collection of unbounded subsets in the plane is \emph{eventually disjoint} if the intersection of any two is bounded. Recall that a generalized unbounded decomposition of $P$ is a collection of eventually disjoint unbounded continua that covers $P$.

\begin{definition}
Let $\cD$ be a collection of eventually disjoint unbounded continua in the plane $P$ with a corresponding set of ends $\cE$. Suppose $\kappa, \lambda, \mu$ is a triple of distinct ends in $\cE$. Choose a disc $D$ that distinguishes these ends and define
\[ \langle \kappa, \lambda, \mu \rangle  = \langle \kappa(D), \lambda(D), \mu(D) \rangle . \]
\end{definition}

\begin{proposition}
Let $\cD$ be a collection of eventually disjoint unbounded continua in the plane $P$ with a corresponding set of ends $\cE$. The function $\langle \cdot, \cdot, \cdot \rangle $ defines a circular order on $\cE$.
\end{proposition}
\begin{proof}
All we need to check is that the order does not depend on the choice of disc.

Let $\kappa, \lambda, \mu$ be a triple in $\cE$ and let $D, D'$ be bounded open discs that distinguish $\kappa, \lambda, \mu$. Let $K = \kappa(D)$, $L = \lambda(D)$, and $M = \mu(D)$ and define $K'$, $L'$, and $M'$ similarly. Without loss of generality we may assume that $D \subset D'$ so that $K' \subset K$, etc.

Suppose that $\langle K, L, M \rangle  = +1$. Then there is an arc $\gamma$ from $K$ to $M$ such that $L$ is on the positive side of $\gamma$ and hence so is $L'$. So we can choose an arc $c$ from the positive side of $\gamma$ to $L'$. Let $U$ and $V$ be disjoint connected open neighborhoods of $K$ and $M$ that avoid $L \cup c$. Choose an arc $\gamma_0$ contained in $U$ that runs from the initial point of $\gamma$ to $K'$ and an arc $\gamma_1$ contained in $V$ that runs from the terminal point of $\gamma$ to $M'$. Then $\gamma' = \gamma_0 \cup \gamma \cup \gamma_1$ is an arc from from $K'$ to $M'$. The arc $c$ exhibits that $L' \subset C^+(\gamma')$ so $\langle K', L', M' \rangle  = +1$ as desired.
\end{proof}

Let $S$ be a circularly ordered set. Two pairs $x, y$ and $z, w$ of points in $S$ are \emph{linked} if either $z \in (x, y)$ and $w \in (y, x)$ or $w \in (x, y)$ and $z \in (y, x)$. Two subsets $A$ and $B$ in $S$ are linked if there are $x, y \in A$ and $z, w \in B$ that are linked. A subset $A \subset S$ \emph{separates} the subsets $B, C \in S$ if there are points $a, a' \in A$ such that $B \subset (a, a')$ and $C \subset (a', a)$. Note that this is not the same as topological separation in $S$ with the order topology.

\begin{proposition}\label{prop:DetectSeparation}
Let $K, L \subset P$ be disjoint unbounded continua in the plane. Then $\cE(K)$ and $\cE(L)$ do not link in the canonical circular order on $\cE(\{K, L\})$.

Let $K, L, M \subset P$ be disjoint unbounded continua in the plane. Then $K$ separates $L$ from $M$ if and only if $\cE(K)$ separates $\cE(L)$ from $\cE(M)$ in the canonical circular order on $\cE(\{K, L, M\})$.
\end{proposition}
\begin{proof}
Suppose that, contrary to the first statement, there are ends $\kappa_1, \kappa_2 \in \cE(K)$ and $\lambda_1, \lambda_2 \in \cE(L)$ such that $\langle \kappa_1, \lambda_1, \kappa_2 \rangle  = +1$ and $\langle \kappa_1, \lambda_2, \kappa_2 \rangle  = -1$. Let $D$ be a bounded open disc distinguishing the $\kappa_i(D)$ and $\lambda_j(D)$ and choose an oriented arc $\gamma$ from $\kappa_1(D)$ to $\kappa_2(D)$ that avoids $L$. Then $\lambda_1(D)$ and $\lambda_2(D)$ are on the same side of $\gamma$ since they are both contained in $L$, and $L$ is connected and disjoint from $\gamma$.

One direction in the second statement follows from the first. Suppose that $K$ does not separate $L$ from $M$. Then we can choose an arc $\gamma$ from $L$ to $M$ disjoint from $L$ and there is an order-preserving bijection between $\cE(L \cup \gamma \cup M)$ and $\cE(L) \cup \cE(M)$. But $L \cup \gamma \cup M$ and $K$ are disjoint so by the first statement their ends do not link, hence $\cE(K)$ does not separate $\cE(L)$ from $\cE(M)$.

It remains to show that if $K$ separates $L$ from $M$ then there are ends $\kappa_1, \kappa_2 \in \cE(K)$, $\lambda \in \cE(L)$, and $\mu \in \cE(M)$ such that $\lambda \in (\kappa_1, \kappa_2)$ and $\mu \in (\kappa_2, \kappa_1)$. Let $\gamma$ be an arc from $L$ to $M$ and let $\gamma'$ be the minimal connected sub-arc that contains $\gamma \cap K$. Set 
\[ K' := K \cup \gamma' \]
and
\[ K'_\pm := K' \cap \overline{C^\pm(\gamma; L, M)}. \]
Note that the $K'_\pm$ are both closed and connected and neither one separates $L$ from $M$. Also, $K'_+ \cap K'_-$ is connected so in order for $K' = K'_+ \cup K'_-$ to separate $L$ from $M$ both $K'_+$ and $K'_-$ must be unbounded by \cite{Wilder}, Theorem II.5.23. \footnote{If $x$ and $y$ are points of $S^2$ which are not separated by either of the closed sets $A$ and $B$, and $A \cap B$ is connected,then $x$ and $y$ are not separated by $A \cup B$.} So we can choose $\kappa_1$ to be any end of $L$ such that $\kappa_1(D) \subset L'_+$ for sufficiently large $D$ and similarly for $\kappa_2$.
\end{proof}

\subsection{Useful properties of the circular order}
We'll collect a few lemmas that will be useful later on.

\begin{lemma} \label{lem:CanSeparate}
Let $K \subset P$ be an unbounded continuum in the plane and $\kappa_1, \kappa_2, \kappa_3 \in \cE(K)$ with $\kappa_2 \in (\kappa_1, \kappa_3)$. Suppose $U \subset P$ is an connected, bounded open set that distinguishes $\kappa_1, \kappa_2, \kappa_3$. Then $U$ distinguishes $\kappa_2$ from any end in $(\kappa_3, \kappa_1)$.
\end{lemma}
\begin{proof}
Suppose that $U$ does not distinguish $\kappa_2$ from $\kappa_4 \in (\kappa_3, \kappa_1)$. Then $\kappa_2(U) = \kappa_4(U)$. By Proposition~\ref{prop:DetectSeparation}, $\kappa_2$ separates $\kappa_1(U)$ from $\kappa_3(U)$, a contradiction.
\end{proof}

\begin{lemma} \label{lem:AvoidBoundedSet}
Let $K, L$ and $M_1, M_2, ... M_n$ be disjoint mutually nonseparating unbounded continua in the plane such that $M_1 \cup ... \cup M_n$ does not separate $K$ and $L$ in the circular order on $\{K, L, M_1, M_2, ... M_n\}$. Let $A$ be a compact set not intersecting $K$ and $L$. Then there is an arc $\gamma$ from $K$ to $L$ that avoids the $M_i$ and $A$.
\end{lemma}
\begin{proof}
It suffices to show that there is an arc avoiding the $M_i$ and $A' = A \cap \overline{C(K, L, M_1, ..., M_n)}$. Let $B$ be a connected compact set intersecting $A'$ and all of the $M_i$ and set $M = A \cup B \cup \bigcup_i M_i$. The ends of $M$ do not separate the ends of $K$ and $L$ so the claim follows from Proposition~\ref{prop:DetectSeparation}.
\end{proof}

\begin{lemma}\label{lem:OutermostArcs}
Let $K, M \subset P$ be disjoint closed, unbounded sets in the plane and let $S \subset P$ be an embedded circle with the usual orientation that intersects both $K$ and $M$. Then there is a unique ``outermost'' sub-arc $\gamma_o \subset S$ that runs from from $K$ to $M$ with orientation inherited from $S$ such that if $\gamma'$ is a component of $S \cap C^+(\gamma_o)$ then the endpoints of $\gamma'$ are both in $K$ or both in $L$.
\end{lemma}
\begin{proof}
Note that $S \setminus (K \cup L)$ is a countable collection of oriented open arcs. Let $\Gamma$ be the collection of closures of these arcs. We can partition
\[ \Gamma = \Gamma_K \cup \Gamma_M \cup \Gamma_{K,M} \cup \Gamma_{M,K} \]
where
\[ \Gamma_K = \{ \text{sub-arcs from } K \text{ to itself} \}, \]
\[ \Gamma_M = \{ \text{sub-arcs from } M \text{ to itself} \}, \]
\[ \Gamma_{K,M} = \{ \text{sub-arcs from } K \text{ to } M \}, \]
and
\[ \Gamma_{M,K} = \{ \text{sub-arcs from } M \text{ to } K \}. \]

Let $\Gamma_s = \Gamma_{K,M} \cup \Gamma^-_{M,K}$, where the minus means to reverse the orientation of the arcs. By Lemma~\ref{lem:ArcSides}, there is a well-defined linear order on $\Gamma_s$, where we define $\gamma < \gamma'$ if $\gamma' \subset C^+(\gamma)$. We will show that $\Gamma_s$ has a maximum element and that this maximum element lies in $\Gamma_{K,M}$.

Suppose $\Gamma_s$ has no maximum element. Then we can find a sequence $\gamma_0 < \gamma_1 < \gamma_2, ...$ in $\Gamma_s$ such that for each $\gamma \in \Gamma_s$ there exists $j$ such that $\gamma < \gamma_j$. Then let $\gamma_\infty$ be the Hausdorff limit of a convergent subsequence of the $\gamma_i$, which exists because $S$ is compact. Although $\gamma_\infty$ is a sub-arc of $S$, it is not necessarily an element of $\Gamma$ since $\mathring{\gamma}_\infty$ might intersect $K$ or $M$. However there is at least one sub-arc $\gamma \subset \gamma_\infty$ from $K$ to $M$ or $M$ to $K$. Then $\gamma \in \Gamma_s$ and $\gamma_i < \gamma$ for all $i$, a contradiction. So there must be a maximum element $\gamma_o \in \Gamma$.

We now turn to showing that $\gamma_o \in \Gamma_{K, M}$ rather than $\Gamma_{M, K}^-$. Let
\[ K' = K \cup (\bigcup_{\gamma \in \Gamma_K} \gamma) \]
and
\[ M' = M \cup (\bigcup_{\gamma \in \Gamma_M} \gamma). \]
Then $K'$ and $L'$ are still disjoint closed, unbounded sets so $C := C^+(\gamma_o; K', L')$ is unbounded (by Lemma~\ref{lem:SeparatingArc}). Note that $C$ is now disjoint from $S$, so if $\gamma_o$ were oriented from $M$ to $K$ then $C$ would be contained in the bounded component of $P \setminus S$, which is impossible.
\end{proof}

\begin{lemma}\label{lem:OutermostArcsDisjoint}
Let $K, L, M, N \subset P$ be disjoint closed, unbounded, mutually nonseparating sets in the plane such that $K, L, M, N$ is a positively oriented quadruple, and let $S \subset P$ be an embedded circle intersecting all four sets. Then the outermost arc from $K$ to $L$ is disjoint from the outermost arc from $M$ to $N$.
\end{lemma}
\begin{proof}
Suppose otherwise. Without loss of generality we have an arc $\gamma: [0, 1] \to P$ such that $\gamma([0, y])$ is an arc from $K$ to $L$ and $\gamma([x, 1])$ is an arc from $M$ to $N$ with $x < y$. Let $\mu \in \cE(M)$. Let $U$ be a connected neighborhood of $M$ that is disjoint from $K$, $L$, $N$, and $\gamma([y, 1])$. Then since $\langle K, M, L \rangle  = -1$ we see that $M' = \mu(\gamma)$ is on the negative side of $\gamma([0, y])$ in $(K \cup L)^c)$ and we can find an arc $c$ from the negative side of $\gamma([0, y])$ to $M'$ that avoids $K$, $L$, and $\gamma([y, 1])$. But $c$ certifies that $M$ is on the negative side of $\gamma([0,1])$ in $(K \cup N)^c$ which contradicts $\langle K, M, N \rangle  = +1$.
\end{proof}

\section{Universal circles and the end compactification} \label{sec:UCandEC}
In this section we show that a generalized unbounded decomposition of the plane gives rise to a natural compactification of the plane as a closed disc. Applying this to the special case of the orbit space $P$ of a quasigeodesic flow and the decompositions $\cD^\pm$ or $\cD$ completes the proofs of the Compactification Theorem and the results of Section~\ref{sec:ClosedOrbits}.

\subsection{Universal circles for unbounded decompositions}\label{sec:UniversalCircles}
We will now show that if $\cD$ is an uncountable, eventually disjoint collection of closed sets in the plane then there is a natural space $S_u(\cD)$ constructed from $\cE(\cD)$ that is homeomorphic to $S^1$.

\begin{lemma} \label{lem:Separable}
Let $\cD$ be an eventually disjoint collection of unbounded continua in the plane. Then the set of ends $\cE(\cD)$ with the canonical circular order is separable in the order topology.
\end{lemma}
\begin{proof}
Let $\{D_i\}_{i = 1}^\infty$ be an exhaustion of the plane by nested bounded open discs. For each $i$, consider
\[ P^i := \cup_{K \in \cD} K_{D_i} \]
where $K_{D_i}$ is the union of the unbounded components of $K \setminus D_i$. Since a subspace of a separable metric space is separable we can choose for each $i$ a countable set $\{p^i_j\}_{j = 1}^\infty$ that is dense in $P^i$. Now for each $i, j$ let $K^i_j \in \cD^+$ be the decomposition element containing $p^i_j$ and choose an end $\kappa^i_j$ such that $\kappa^i_j(D_i)$ is the component of $K^i_j \setminus D_i$ that contains $p^i_j$. Now let
\[ \cE' = \{ \kappa^i_j \}_{i, j}. \]

We will show that $\cE'$ is dense in $\cE$. Indeed, let $\kappa, \mu \in \cE(\cD)$ and assume that $(\kappa, \mu)$ contains at least one end, $\lambda$. Choose $D$ large enough to distinguish $\kappa$, $\lambda$, and $\mu$ and let
\[ K = \kappa(D), \]
\[ M = \mu(D). \]

Let $\gamma$ be an arc from $K$ to $M$ and choose $i$ large enough so that $D_i$ contains $D$ and $\gamma$. Set $U := C^+(\gamma; K, M)$.
Now $U \cap K_{D_i}$ is an open subset of $D_i$ and it is nonempty because it contains $\lambda(D_i)$. So $p^i_j$ is contained in $U$ for some $j$. The corresponding end $\kappa^i_j$ is contained in $(\kappa, \mu)$ since $\kappa^i_j(D_i)$ is contained in $U$.
\end{proof}

Recall that a \emph{gap} in a circularly ordered set $S$ is an ordered pair of distinct elements $x, y \in S$ such that the open interval $(x, y) \subset S$ is empty.

\begin{lemma} \label{lem:CountablyManyGaps}
Let $\cD$ be an eventually disjoint collection of unbounded continua in the plane. Then the set of ends $\cE(\cD)$ has at most countably many gaps.
\end{lemma}
\begin{proof}
Let $\{D_i\}_{i=1}^\infty$ be an exhaustion of $P$ by bounded open discs and let $S_i^j = \partial D_i$ for all $i > j$ ($j$ is a dummy variable, i.e. we want to keep track of all of the circles outside of $j$ for each $j$). If $(\kappa, \lambda)$ form a gap then let $n$ be the first integer such that $\kappa(D_n) \neq \lambda(D_n)$ and let $k$ be the first integer such that $S_k^n$ intersects both $\kappa(D_n)$ and $\lambda(D_n)$. We can associate to the gap $(\kappa, \lambda)$ the open interval $U_{\kappa, \lambda} \subset S_k^n$ whose closure is the outermost arc from $\kappa(D_n)$ to $\lambda(D_n)$ (see Lemma~\ref{lem:OutermostArcs}), and distinct gaps correspond to disjoint open intervals in $\{S_j^i\}_{0 < i < j}$ (see Lemma~\ref{lem:OutermostArcsDisjoint}). There can only be finitely many such open intervals and thus only countably many gaps.
\end{proof}

\begin{construction}
Let $\cD$ be a generalized unbounded decomposition of the plane $P$. Then the corresponding set of ends $\cE = \cE(\cD)$ is an uncountable circularly ordered set with countably many gaps that is separable in the order topology. We can construct a \emph{universal circle} $S_u(\cD) = \widehat{\overline{\cE}}$ as in Construction~\ref{con:UCforCO}. Note that any homeomorphism of $P$ preserving the decomposition $\cD$ induces a natural homeomorphism on $S_u$.
\end{construction}

\begin{example} \label{ex:NoRightmostEnd}
It is appealing to think that the image of the set of ends of a decomposition element in the universal circle is closed but this is not generally true. For example, the set in Figure~\ref{fig:NoRightmostEnd} has no rightmost end.

\begin{figure}[h]
	\centering
		\includegraphics{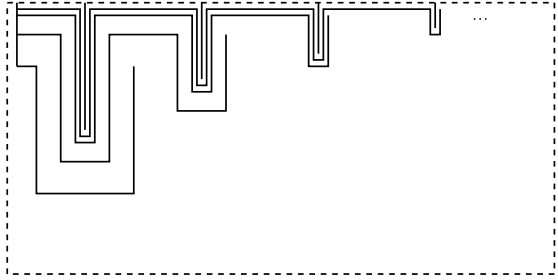}
	\caption{An unbounded continuum with no rightmost end.}\label{fig:NoRightmostEnd}
\end{figure}

\end{example}

These results fill a gap in the literature. See \cite{Calegari} and compare with with \cite{Thurston}, Section 2.1.3. Note that not all circularly ordered sets admit a universal circle constructed this way. For example, take $S^1 \times \{0, 1\}$ where each $S^1 \times \{i\}$ has the usual circular order and $(s, 1)$ is immediately counterclockwise to $(s, 0)$ for all $s \in S^1$. This has uncountably many gaps and is not 2nd countable.

\subsection{The end compactification} \label{sec:EndCompactification}
Throughout this section $\cD$ will be an unbounded generalized decomposition of the plane $P$. Endow the set of ends $\cE := \cE(\cD)$ with the canonical circular order and let $S := S_u(\cD)$ be the corresponding universal circle. If $A \subset \cE$ we will denote by $\widehat{A}$ the corresponding set in $S$.

\begin{definition}
A set $K \subset P$ is said to be \emph{subordinate} to the decomposition $\cD$ if $K = \kappa(A)$ for some end $\kappa \in \cE(\cD)$ and bounded open set $A \subset P$.
\end{definition}

If $K = \kappa(A)$ is subordinate to $\cD$ and $K' \in \cD$ is the decomposition element containing $K$ then we can identify $\cE(K)$ with the subset of $\cE(K')$ consisting of ends $\kappa'$ such that $\kappa'(A) = K$. We will thus write $\cE(K) \subset \cE(\cD)$.

\begin{construction}
We will construct a space $\overline{P}$, called the \emph{end compactification of $P$ with respect to $\cD$}. If there is a possibility of confusion about which decomposition we are working with then we will write $\overline{P_\cD}$.

As a set, $\overline{P} := P \cup S$.

Suppose that $K$ and $L$ are disjoint sets subordinate to $\cD$ and $\gamma$ is an arc from $K$ to $L$. Let $I$ be the maximal open interval in $S$ running from $\widehat{\cE(K)}$ to $\widehat{\cE(L)}$ with positive orientation. The set
\[ O(K, L, \gamma) := I \cup C^+(\gamma; K, L) \]
is called the \emph{peripheral set determined by $K$, $L$, and $\gamma$}.

The topology on $\overline{P}$ is the coarsest topology containing all open sets in $P$ and all peripheral sets.
\end{construction}

Note that it is possible that $O(K, L, \gamma) \cap S = \emptyset$, i.e. that the interval $I$ above may be empty. For example, take $K$ to be the set in Example~\ref{ex:NoRightmostEnd} and $L$ its mirror image.

In the following discussion we will often need to find subordinate elements with ends in a specified set. We'll start with some observations in this direction that will be used casually in the sequel.

If $(a, b) \subset S$ is any open interval then we can find a set $L$ subordinate to $\cD$ such that $\cE(L) \subset (a, b)$. Indeed, since $\cE(\cD)$ has dense image in $S$ we can find ends $\kappa, \lambda, \mu \in \cE(\cD)$ with image in $(a, b)$ such that $\langle \kappa, \lambda, \mu \rangle  = +1$. Let $D$ be a bounded open disc that distinguishes these three ends. Then by Lemma~\ref{lem:CanSeparate}, $L := \lambda(D)$ satisfies our requirement. In addition, by choosing $D$ sufficiently large we can ensure that $L$ is disjoint from any specified bounded set.

In addition, if $K_1, ... K_n$ are subordinate sets with ends outside of $(a, b)$ then we can choose $L$ disjoint from $\bigcup K_i$. Indeed, if $K_i = \kappa_i(D_i)$ then choose $D$ to contain $\bigcup_i D_i$.

\begin{proposition}
Let $\cD$ be a generalized unbounded decomposition of the plane $P$ and let $\overline{P}$ be the end-compactification of $P$ with respect to $\cD$. Then
\begin{enumerate}
 \item $\overline{P}$ has a basis consisting of open sets in $P$ and peripheral sets,
 \item $\overline{P}$ is 1st countable, and
 \item the inclusion maps $P \hookrightarrow \overline{P}$ and $S_u \hookrightarrow \overline{P}$ are homeomorphisms.
\end{enumerate}
\end{proposition}
\begin{proof}
Given $p \in S \subset \overline{P}$ we will construct a sequence of peripheral sets $U_i$ that is eventually contained in any peripheral set containing $p$. 

Fix an exhaustion of the plane by bounded sets $D_i$ and a sequence of open intervals $(a_i, b_i)$ in $S$ such that $[a_{i+1}, b_{i+1}] \subset (a_i, b_i)$ for all $i$ and $\bigcap_i (a_i, b_i) = p$. Let $K_0$ and $L_0$ be subordinate sets with ends in $(a_0, a_{1})$ and $(b_{1}, b_0)$ respectively and let $\gamma_0$ be an arc from $K_0$ to $L_0$. Set 
\[ U_0 = O(K_0, L_0, \gamma_0). \]

For each $i$ choose subordinate sets $K_i$ and $L_i$ with ends in $(a_i, a_{i+1})$ and $(b_{i+1}, b_i)$ respectively that are disjoint from $D_i$, $K_{i-1}$, $L_{i-1}$, and $\gamma_{i-1}$. By Lemma~\ref{lem:AvoidBoundedSet} we can find an arc $\gamma_i$ from $K_i$ to $L_i$ that is contained in $U_{i-1}$. Setting
\[ U_i = O(K_i, L_i, \gamma_i) \]
we have $U_i \subset U_{i-1}$ and $U_i \cap D_{i} = \emptyset$.

Now suppose that $U$ is a peripheral set containing $p$. Then the $U_i$ are eventually contained in $U$. Indeed, $U = O(\kappa(D'), \lambda(D''), \gamma)$ for some $\kappa$, $\lambda$, $D'$, $D''$, and $\gamma$. Let $D$ be a compact disc containing $D', D'',$ and $\gamma$. Choose $I$ large enough so that $(a_I, b_I) \subset U \cap V$ and $D_I \supset D$. Then $U_i \subset U$ for all $i > I$.

To prove (1) it suffices to show that if $U$ and $V$ are either peripheral or open in $P$ and $p \subset U \cap V$ then there is a set $W$ that is either peripheral or open in $P$ such that $p \in W \subset U \cap V$. If $p \in P$ then this is obvious, so assume that $p \in S$ and $U$ and $V$ are both peripheral. Simply choose $i$ large enough so that $U_i \subset U \cap V$.

Now that we know (1), the construction above yields a countable basis about any point in $S$ and (2) follows.

For (3), it is clear that the inclusion $P \hookrightarrow \overline{P}$ is a homeomorphism.

Let $i:S_u \hookrightarrow \overline{P}$. It is clear that the preimage of an open set under $i$ is open. On the other hand, suppose that $U \in S_u$ is open and let $p \in U$. Choose points $a, b, c, d \in S_u$ such that $a, b, p, c, d$ is positively oriented and let $K$ and $L$ be subordinate sets whose ends lie in $(a, b)$ and $(c, d)$ respectively. Let $\gamma$ be an arc from $K$ to $L$. Then $O(K, L, \gamma) \cap \partial \overline{P}$ is contained in $i(U)$ so $i(U)$ is open.
\end{proof}

We can now make sense of the term ``end compactification.''

\begin{proposition}
Let $\cD$ be a generalized unbounded decomposition of the plane $P$. Then the end compactification $\overline{P}$ with respect to $\cD$ is compact.
\end{proposition}
\begin{proof}
Let $\cU$ be an open cover of $\widetilde{P_\cD}$. Since $S$ is compact we can find a finite subcollection $\{U_1, U_2, ..., U_n\} \subset \cU$ that covers $S$. After reordering and taking a refinement we can assume that each $U_i$ is a peripheral set $U_i = O(K_i, L_i, \gamma_i)$ such that $K_i, L_{i-1}, K_{i+1}, L_i$ is positively ordered for all $i$ (mod $n+1$). Further, we can assume that the arcs $\gamma_i$ do not intersect the $K_i$ and $L_i$ and that $\gamma_i$ intersects $\gamma_{i+1}$ only once, transversely, for all $i$. See Figure~\ref{fig:Compactness}.

\begin{figure}[h]
	\centering
	\labellist
	\small\hair 2pt
	\pinlabel $\overline{\orb}$ at 13 134
	\pinlabel $K_i$ [tr] at 98 39
	\pinlabel $L_{i-1}$ [bl] at 122 54
	\pinlabel $K_{i+1}$ [tl] at 116 111
	\pinlabel $L_i$ [r] at 94 129
	\pinlabel $\gamma_i$ [r] at 80 75
	\pinlabel $\gamma_{i-1}$ [br] at 62 47
	\pinlabel $\gamma_{i+1}$ [tr] at 68 91
	
	\pinlabel $...$ at 52 128
	\pinlabel $...$ at 65 19
	
	\endlabellist
		\includegraphics{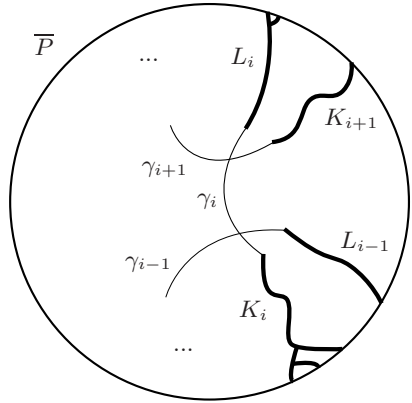}
	\caption{A refinement covering $S$.} \label{fig:Compactness}
\end{figure}

We can concatenate the sub-arcs between intersections of the $\gamma_i$ to form a simple closed curve $\gamma$. Points on the negative side of $\gamma$ are contained in at least one of the $U_i$ and the positive side of $\gamma$ is compact. So we can find a finite sub-cover of $\cU$ that covers this compact piece and the rest is covered by the $U_i$.
\end{proof}

We use the abbreviation $cl_X(Y)$ for the closure of $Y$ in $X$.
\begin{lemma}
Let $K$ be subordinate to the decomposition $\cD$. Then
\[ cl_{\overline{P}}(K) = K \cup cl_{S}(\widehat{\cE(K)}). \]
\end{lemma}
\begin{proof}
It is clear that $cl_{\overline{P}}(K) \cap P = K$, so let $p \in S$. If $p$ is not in $cl_{S}(\widehat{\cE(K)})$ then we can find a peripheral set containing $p$ that does not intersect $K$, so $p \notin cl_{\overline{P}}(K)$.
\end{proof}

\begin{theorem}
Let $\cD$ be a generalized unbounded decomposition of the plane $P$ and let $S$ be the corresponding universal circle. Then the end compactification $\overline{P}$ is homeomorphic to a closed disc with boundary $S$. Additionally, any homeomorphism of $P$ that preserves $\cD$ extends to a homeomorphism of $\overline{P}$.
\end{theorem}

\begin{proof}
The second statement follows from the fact that the image of a peripheral set under a $\cD$-preserving homeomorphism is again a peripheral set. To prove that $\overline{P}$ is a closed disc we will use the following theorem. An arc $\gamma$ with endpoints $a,b$ is said to \emph{span} a set $S \subset X$ if $a, b \in S$ and $\mathring{\gamma} \subset X \setminus S$.

\begin{theorem}[Zippin, \cite{Wilder}, Theorem II.5.1]
Let $X$ be a connected, compact, 1st countable Hausdorff space. Suppose that there is a 1-sphere $S \in X$ such that there exists an arc in $X$ spanning $S$, every arc that spans $S$ separates $X$, and no closed proper subset of an arc spanning $S$ separates $X$. Then $X$ is homeomorphic to a closed 2-disc with boundary $S$.
\end{theorem}

It is clear that $\overline{P}$ is connected and Hausdorff and we have already shown that it is 1st countable and compact. Let's check the remaining conditions.

{\bf Existence of a spanning arc:}
Let $p \in S$. As in the proof of 1st countability we can find a sequence $\{U_i\}$ of nested open neighborhoods of $p$ such that $\cap U_i = p$ and $A_i = U_i \setminus \overline{U_{i+1}}$ is an open disc. Choose a point $p_i \in \gamma_i$ for each $i$ and let $c_i$ be an arc lying in $A_i$ that connects $p_i$ to $p_{i+1}$. The concatenation of these arcs is a proper ray since the $c_i$ are eventually disjoint from any bounded set in $P$. The closure of this ray is an arc from $c_0$ to $p$. Construct two such rays and connect their endpoints with an arc in $P$.

{\bf Spanning arcs separate:}
Let $\gamma$ be an arc spanning $S$ with initial point $a$ and terminal point $b$. Note that $\mathring{\gamma}$ is a properly embedded curve in $P$ so it separates $P$ into two unbounded discs by the Jordan curve theorem. Let $D_+$ and $D_-$ be the discs on the positive and negative sides of $\gamma$ respectively. We will show that $\gamma$ separates $\overline{P}$ into
\[ D'_+ = D_+ \cup (a,b) \]
and
\[ D'_- = D_- \cup (b,a). \]
Note that this statement is a little more than what we need: it also says that the orientations work out as expected.

Let $U_a$ and $U_b$ be peripheral sets containing $a$ and $b$ respectively. We may assume that $(x, y) = U_a \cap S$ and $(z, w) = U_b \cap S$ are disjoint. Since the part $\gamma'$ of $\gamma$ lying outside of $U_a \cup U_b$ is bounded we can find a subordinate set $K$ whose ends lie in $(y, z)$ that is disjoint from $K_a, L_a, \mu_a, K_b, L_b, \mu_b$ and $\gamma'$ and therefore disjoint from $\gamma$. Similarly find $L$ with ends in $(w, x)$ that is disjoint from $\gamma$.

Suppose that $\gamma$ did not separate $K$ from $L$. Then we could find an arc $\mu$ from $K$ to $L$ disjoint from $\gamma$. But then $b \in O(K, L, \mu)$ while $a \notin O(K, L, \mu)$ implying that $\gamma$ must intersect $\partial O(K, L, \mu)$ contrary to our assumption. So $\gamma$ must separate $K$ from $L$.

Now let $\mu$ be an arc from $K$ to $L$ that intersects $\gamma$ once and transversely. Then $K$ is on the positive side of $\gamma$ while $L$ is on the negative side. Note that $\overline{K}$ and $\overline{L}$ are connected since $K$ and $L$ are connected, so $D'_+$ and $D'_-$ are connected.

A similar argument shows that every point in $(a, b)$ has a neighborhood contained in $D'_+$ and every point in $(b, a)$ has a neighborhood contained in $D'_-$. This suffices to show that $D'_+$ and $D'_-$ are separated.

{\bf No subset 	of a spanning arc separates:}
This is now obvious.
\end{proof}

\end{document}